\documentclass[10pt]{amsart}
\usepackage{amssymb,MnSymbol}
\usepackage{amsthm,amsmath,cite}
\usepackage{stmaryrd}

\title[]{Discrete integrals based on comonotonic modularity}

\author{Miguel Couceiro}
\address{Lamsade, University Paris-Dauphine \\
Place du Mar\'echal de Lattre de Tassigny, F-75775 Paris cedex 16, France}%
\email{miguel.couceiro[at]dauphine.fr}

\author{Jean-Luc Marichal}
\address{Mathematics Research Unit, FSTC, University of Luxembourg \\
6, rue Coudenhove-Kalergi, L-1359 Luxembourg, Luxembourg}%
\email{jean-luc.marichal[at]uni.lu}

\date{June 6, 2013}

\begin{document}

\theoremstyle{plain}
\newtheorem{theorem}{Theorem}
\newtheorem{lemma}[theorem]{Lemma}
\newtheorem{proposition}[theorem]{Proposition}
\newtheorem{corollary}[theorem]{Corollary}
\newtheorem{fact}[theorem]{Fact}
\newtheorem*{main}{Main Theorem}

\theoremstyle{definition}
\newtheorem{definition}[theorem]{Definition}
\newtheorem{example}[theorem]{Example}

\theoremstyle{remark}
\newtheorem*{conjecture}{Conjecture}
\newtheorem{remark}{Remark}
\newtheorem{claim}{Claim}

\newcommand{\N}{\mathbb{N}}
\newcommand{\R}{\mathbb{R}}
\newcommand{\B}{\mathbb{B}}
\newcommand{\bfzero}{\mathbf{0}}
\newcommand{\bfone}{\mathbf{1}}
\newcommand{\bfx}{\mathbf{x}}
\newcommand{\vect}[1]{\ensuremath{\mathbf{#1}}} 

\begin{abstract}
It is known that several discrete integrals, including the Choquet and Sugeno integrals as well as some
of their generalizations, are comonotonically modular functions.
Based on a recent description of the class of comonotonically modular functions,
we axiomatically identify more general families of discrete integrals that are comonotonically modular,
including signed Choquet integrals and symmetric signed Choquet integrals as well as natural
extensions of Sugeno integrals.
\end{abstract}



\maketitle

\section{Introduction}

Aggregation functions arise wherever merging information is needed: applied and pure mathematics (probability, statistics, decision theory, functional equations), operations research, computer science, and many applied fields (economics and finance, pattern recognition and image processing, data fusion, etc.). For recent references, see Beliakov et al.~\cite{BelPraCal07} and Grabisch et al.~\cite{GraMarMesPap09}.

Discrete Choquet integrals and discrete Sugeno integrals are among the best known functions in aggregation theory,
mainly because of their many applications, for instance, in decision making
(see the edited book \cite{GraMurSug00}). More generally, signed Choquet integrals,
which need not be nondecreasing in their arguments, and the Lov\'asz extensions of pseudo-Boolean functions,
which need not vanish at the origin, are natural extensions of the Choquet integrals and have been thoroughly
investigated in aggregation theory. For recent references, see, e.g., \cite{CarCouGioMar11,CouMar11}.


The class of $n$-variable Choquet integrals has been axiomatized independently by means of two noteworthy
aggregation properties, namely comonotonic additivity (see, e.g., \cite{deCBol92}) and
horizontal min-additivity (originally called ``horizontal additivity'', see \cite{BenMesViv02}).
Function classes characterized by these properties have been recently described by the authors \cite{CouMar11}.
Quasi-Lov\'asz extensions, which generalize signed Choquet integrals and Lov\'asz extensions
by transforming the arguments by a 1-variable function, have also been recently investigated by the authors
\cite{CouMar11b} through natural aggregation properties.

Lattice polynomial functions and quasi-Sugeno integrals generalize the notion of Sugeno integral
\cite{CouMar09,CouMar10,CouMar1,CouMar2,CouMar0}: the former
by removing the idempotency requirement and the latter also by transforming arguments by a 1-variable function. Likewise, these functions have been axiomatized by means of well-known properties such as comonotonic maxitivity and comonotonic minitivity.

All of these classes share the feature that its members are comonotonically modular.
These facts motivated a recent study that led to a description of comonotonically modular functions \cite{CouMar11b}.
In this paper we survey these and other results and present a somewhat typological study of the
vast class of comonotonically modular functions, where we identify several families of discrete integrals
within this class using variants of homogeneity as distinguishing feature.

The paper is organized as follows. In Section \ref{signedChoq} we recall basic notions and terminology
 related to the concept of signed Choquet integral and present some preliminary characterization results.
 In Section~\ref{comodularity} we survey several results that culminate in a description of comonotonic modularity and
 establish connections to other well studied properties of aggregation functions. These results are then used in
 Section~\ref{COMDinteg} to provide characterizations of the various classes of functions considered in the
 previous sections as well as of classes of functions extending Sugeno integrals.

We employ the following notation throughout the paper. The set of permutations on $X=\{1,\ldots,n\}$ is denoted by $\mathfrak{S}_X$. For every $\sigma\in\mathfrak{S}_X$, we define
$$
\R^n_{\sigma} ~=~ \{\bfx=(x_1,\ldots,x_n)\in\R^n : x_{\sigma(1)}\leqslant\cdots\leqslant x_{\sigma(n)}\}.
$$
Let $\R_+=\left[0,+\infty\right[$ and $\R_-=\left]-\infty,0\right]$. We let $I$ denote a
nontrivial (i.e., of positive Lebesgue measure) real interval, possibly unbounded. We also introduce the notation $I_+=I\cap\R_+$, $I_-=I\cap\R_-$, and $I_{\sigma}^n=I^n\cap\R_{\sigma}^n$. For every $S\subseteq X$, the symbol $\mathbf{1}_S$ denotes the
$n$-tuple whose $i$th component is $1$, if $i\in S$, and $0$, otherwise. Let also $\mathbf{1}=\mathbf{1}_X$ and
$\mathbf{0}=\mathbf{1}_{\varnothing}$. The symbols $\wedge$ and $\vee$ denote the minimum and maximum functions, respectively. For every
$\bfx\in\R^n$, let $\bfx^+$ be the $n$-tuple whose $i$th component is $x_i\vee 0$ and let $\bfx^-=(-\bfx)^+$.
For every permutation $\sigma\in\mathfrak{S}_X$ and every $i\in X$, we set $S_{\sigma}^{\uparrow}(i)=\{\sigma(i),\ldots,\sigma(n)\}$, $S_{\sigma}^{\downarrow}(i)=\{\sigma(1),\ldots,\sigma(i)\}$, and $S_{\sigma}^{\uparrow}(n+1)=S_{\sigma}^{\downarrow}(0)=\varnothing$.

\section{Signed Choquet integrals}\label{signedChoq}

In this section we recall the concepts of Choquet integral, signed Choquet integral, and symmetric signed Choquet integral. We also recall some axiomatizations of these function classes. For general background, see \cite{CarCouGioMar11,CouMar11,CouMar11b}

A \emph{capacity} on $X=\{1,\ldots,n\}$ is a set function $\mu\colon 2^X\to\R$ such that $\mu(\varnothing)=0$ and $\mu(S)\leqslant \mu(T)$ whenever
$S\subseteq T$.

\begin{definition}
The \emph{Choquet integral} with respect to a capacity $\mu$ on $X$ is the function $C_{\mu}\colon \R_+^n\to\R$ defined as
$$
C_{\mu}(\bfx) ~=~ \sum_{i=1}^nx_{\sigma(i)}\,\big(\mu(S_{\sigma}^{\uparrow}(i))-\mu(S_{\sigma}^{\uparrow}(i+1))\big){\,},\qquad \bfx\in (\R_+^n)_{\sigma},~\sigma\in\mathfrak{S}_X.
$$
\end{definition}

The concept of Choquet integral can be formally extended to more general set functions and $n$-tuples of $\R^n$ as follows. A \emph{signed capacity} on $X$ is a set function $v\colon 2^X\to\R$ such that $v(\varnothing)=0$.

\begin{definition}\label{de:SignedChoquet}
The \emph{signed Choquet integral} with respect to a signed capacity $v$ on $X$ is the function $C_v\colon \R^n\to\R$ defined as
\begin{equation}\label{eq:SignedChoquet}
C_v(\bfx) ~=~ \sum_{i=1}^nx_{\sigma(i)}\,\big(v(S_{\sigma}^{\uparrow}(i))-v(S_{\sigma}^{\uparrow}(i+1))\big){\,},\qquad \bfx\in \R^n_{\sigma},~\sigma\in\mathfrak{S}_X.
\end{equation}
\end{definition}

From (\ref{eq:SignedChoquet}) it follows that $C_v(\mathbf{1}_S)=v(S)$ for every $S\subseteq X$. Thus Eq.~(\ref{eq:SignedChoquet}) can be rewritten as
\begin{equation}\label{eq:SignedChoquet2}
C_{v}(\bfx) ~=~ \sum_{i=1}^nx_{\sigma(i)}\,\big(C_v(\mathbf{1}_{S_{\sigma}^{\uparrow}(i)})-C_v(\mathbf{1}_{S_{\sigma}^{\uparrow}(i+1)})\big){\,},
\qquad \bfx\in \R^n_{\sigma},~\sigma\in\mathfrak{S}_X.
\end{equation}

Thus defined, the signed Choquet integral with respect to a signed capacity $v$ on $X$ is the continuous function $C_v$ whose restriction to each region $\R^n_{\sigma}$ ($\sigma\in\mathfrak{S}_X$) is the unique linear function that coincides with $v$ (or equivalently, the corresponding pseudo-Boolean function $v\colon\{0,1\}^n\to\R$) at the $n+1$ vertices of the standard simplex $[0,1]^n\cap\R^n_{\sigma}$ of the unit cube $[0,1]^n$. As such, $C_v$ is called the \emph{Lov\'asz extension} of $v$.

%

From this observation we immediately derive the following axiomatization of the class of $n$-variable signed Choquet integrals over a real interval $I$. A function $f\colon I^n\to\R$ is said to be a \emph{signed Choquet integral} if it is the restriction to $I^n$ of a signed Choquet integral.

\begin{theorem}[{\cite{CarCouGioMar11}}]
Assume that $0\in I$. A function $f\colon I^n\to\R$ satisfying $f(\mathbf{0})=0$ is a signed Choquet integral if and only if
$$
f(\lambda{\,}\bfx+(1-\lambda){\,}\bfx') ~=~ \lambda{\,}f(\bfx)+(1-\lambda){\,}f(\bfx'),\qquad \lambda\in [0,1],~\bfx,\bfx'\in I^n_{\sigma},~\sigma\in\mathfrak{S}_X.
$$
\end{theorem}

The next theorem provides an axiomatization of the class of $n$-variable signed Choquet integrals based on comonotonic additivity, horizontal min-additivity, and horizontal max-additivity. Recall that two $n$-tuples $\bfx,\bfx'\in I^n$ are said to be \emph{comonotonic} if there exists $\sigma\in \mathfrak{S}_X$ such that $\bfx,\bfx'\in I^n_{\sigma}$. A function $f\colon I^n\to\R$ is said to be \emph{comonotonically
additive} if, for every comonotonic $n$-tuples $\bfx,\bfx'\in I^n$ such that $\bfx+\bfx'\in I^n$, we have
$$
f(\bfx+\bfx')=f(\bfx)+f(\bfx').
$$
A function $f\colon I^n\to\R$ is said to be
\emph{horizontally min-additive} (resp.\ \emph{horizontally max-additive}) if, for every $\bfx\in I^n$ and every $c\in I$ such that $\bfx-\bfx\wedge c\in I^n$ (resp.\ $\bfx-\bfx\vee c\in I^n$), we have
$$
f(\bfx)=f(\bfx\wedge c)+f(\bfx-\bfx\wedge c)\qquad \big(\mbox{resp}.~f(\bfx)=f(\bfx\vee c)+f(\bfx-\bfx\vee c)\big).
$$

\begin{theorem}[\cite{CouMar11}]\label{thm:saddsf687}
Assume $[0,1]\subseteq I\subseteq\R_+$ or $I=\R$. Then a function $f\colon I^n\to\R$ is a signed Choquet integral if and only if the following conditions hold:
\begin{enumerate}
\item[$(i)$] $f$ is comonotonically additive or horizontally min-additive (or horizontally max-additive if $I=\R$).

\item[$(ii)$] $f(c{\,}x\bfone_S)=c{\,}f(x\bfone_S)$ for all $x\in I$ and $c>0$ such that $c{\,}x\in I$ and all $S\subseteq X$.
\end{enumerate}
\end{theorem}

\begin{remark}\label{rem:saddsf687}
It is easy to see that condition $(ii)$ of Theorem~\ref{thm:saddsf687} is equivalent to the following simpler condition: $f(x\bfone_S)=\mathrm{sign}(x){\,}x{\,}f(\mathrm{sign}(x){\,}\bfone_S)$ for all $x\in I$ and $S\subseteq X$.
\end{remark}


We now recall the concept of symmetric signed Choquet integral. Here ``symmetric'' does not refer to invariance under a permutation of variables but rather to the role of the origin of $\R^n$ as a symmetry center with respect to the function values.

\begin{definition}\label{de:SymSignedChoquet}
Let $v$ be a signed capacity on $X$. The \emph{symmetric signed Choquet integral with respect to} $v$
is the function $\check{C}_v\colon \R^n\to\R$ defined as
\begin{equation}\label{eq:SymSignedChoquet}
\check{C}_v(\bfx) ~=~ C_v(\bfx^+)-C_v(\bfx^-),\qquad \bfx\in\R^n.
\end{equation}
\end{definition}

Thus defined, a symmetric signed Choquet integral is an odd function in the sense that $\check{C}_v(-\bfx)=-\check{C}_v(\bfx)$. It is then not difficult to show that the restriction of $\check{C}_v$ to $\R^n_{\sigma}$ is the function
\begin{eqnarray}
\check{C}_v(\bfx) &=& \sum_{i=1}^p
x_{\sigma(i)}\,\big(C_v(\mathbf{1}_{S_{\sigma}^{\downarrow}(i)})-C_v(\mathbf{1}_{S_{\sigma}^{\downarrow}(i-1)})\big)\nonumber\\
&& \null +\sum_{i=p+1}^n
x_{\sigma(i)}\,\big(C_v(\mathbf{1}_{S_{\sigma}^{\uparrow}(i)})-C_v(\mathbf{1}_{S_{\sigma}^{\uparrow}(i+1)})\big),\qquad
\bfx\in\R^n_{\sigma},\label{eq:sdfsfd65dsf}
\end{eqnarray}
where the integer $p\in\{0,\ldots,n\}$ is given by the condition $x_{\sigma(p)}<0\leqslant x_{\sigma(p+1)}$, with the convention that $x_{\sigma(0)}=-\infty$ and $x_{\sigma(n+1)}=+\infty$.


The following theorem provides an axiomatization of the class of $n$-variable symmetric signed Choquet integrals based on horizontal median-additive additivity. Assuming that $I$ is centered at $0$, recall that a function $f\colon I^n\to\R$ is said to be \emph{horizontally median-additive} if, for every $\bfx\in I^n$ and every $c\in I_+$, we have
\begin{equation}\label{eq:sd98f7}
f(\bfx)=f\big(\mathrm{med}(-c,\bfx,c)\big)+f\big(\bfx-\bfx\wedge c\big)+f\big(\bfx-\bfx\vee (-c)\big),
\end{equation}
where $\mathrm{med}(-c,\bfx,c)$ is the $n$-tuple whose $i$th component is the middle value of $\{-c,x_i,c\}$. Equivalently, a function $f\colon I^n\to\R$ is horizontally median-additive if and only if its restrictions to $I^n_+$ and $I^n_-$ are comonotonically additive and
$$
f(\bfx) ~=~ f(\bfx^+) + f(-\bfx^-),\qquad \bfx\in I^n.
$$

A function $f\colon I^n\to\R$ is said to be a \emph{symmetric signed Choquet integral} if it is the restriction to $I^n$ of a symmetric signed Choquet integral.

\begin{theorem}[\cite{CouMar11}]\label{thm:saddsf687a}
Assume that $I$ is centered at $0$ with $[-1,1]\subseteq I$. Then a function $f\colon I^n\to\R$
is a symmetric signed Choquet integral if and only if the following conditions hold:
\begin{enumerate}
\item[$(i)$] $f$ is horizontally median-additive.

\item[$(ii)$] $f(c{\,}x\bfone_S)=c{\,}f(x\bfone_S)$ for all $c,x\in I$ such that $c{\,}x\in I$ and all $S\subseteq X$.
\end{enumerate}
\end{theorem}

\begin{remark}\label{rem:saddsf687a}
It is easy to see that condition $(ii)$ of Theorem~\ref{thm:saddsf687a} is equivalent to the following simpler condition: $f(x\bfone_S)=x{\,}f(\bfone_S)$ for all $x\in I$ and $S\subseteq X$.
\end{remark}

We end this section by recalling the following important formula. For every signed capacity $v$ on $X$, we have
\begin{equation}\label{eq:dual}
C_v(\bfx) ~=~ C_v(\bfx^+)-C_{v^d}(\bfx^-),\qquad \bfx\in\R^n,
\end{equation}
where $v^d$ is the capacity on $X$, called the \emph{dual} capacity of $v$, defined as $v^d(S)=v(X)-v(X\setminus S)$.

\section{Comonotonic modularity}\label{comodularity}

Recall that a function $f\colon I^n\to\R$ is said to be \emph{modular} (or a \emph{valuation}) if
\begin{equation}\label{eq:sdf75}
f(\bfx)+f(\bfx')=f(\bfx\wedge\bfx')+f(\bfx\vee\bfx')
\end{equation}
for every $\bfx,\bfx'\in I^n$, where $\wedge$ and $\vee$ are considered componentwise. It was proved \cite{Top78} that a function $f\colon I^n\to\R$ is modular if and only if it
is \emph{separable}, that is, there exist $n$ functions $f_i\colon I\to\R$ ($i=1,\ldots,n$) such that $f=\sum_{i=1}^nf_i$. In particular, any $1$-variable function $f\colon I\to\R$
is modular.

More generally, a function $f\colon I^n\to\R$ is said to be \emph{comonotonically modular} (or a \emph{comonotonic valuation}) if (\ref{eq:sdf75}) holds for every
comonotonic $n$-tuples $\bfx,\bfx'\in I^n$; see \cite{CouMar11b,MesMes11}. It was shown \cite{CouMar11b} that a
function $f\colon I^n\to\R$ is comonotonically modular if and only if it is \emph{comonotonically separable}, that is,
for every $\sigma\in \mathfrak{S}_X$, there exist functions $f^{\sigma}_i\colon I\to\R$ ($i=1,\ldots,n$) such that
$$
f(\bfx) ~=~ \sum_{i=1}^nf^{\sigma}_i(x_{\sigma(i)}) ~=~ \sum_{i=1}^nf^{\sigma}_{\sigma^{-1}(i)}(x_i),\qquad \bfx\in I^n_{\sigma}.
$$

We also have the following important definitions. For every $\bfx\in\R^n$ and every $c\in\R_+$ (resp.\ $c\in\R_-$) we denote by $[\bfx]_c$ (resp.\ $[\bfx]^c$) the $n$-tuple whose $i$th component is $0$, if $x_i\leqslant c$ (resp.\ $x_i\geqslant c$), and $x_i$, otherwise. Recall that a function $f\colon I^n\to\R$, where $0\in I\subseteq\R_+$, is \emph{invariant under horizontal min-differences} if, for every $\bfx\in I^n$ and every $c\in I$, we have
\begin{equation}\label{eq:HminDif}
f(\bfx)-f(\bfx\wedge c)=f([\bfx]_c)-f([\bfx]_c\wedge c).
\end{equation}
Dually, a function $f\colon I^n\to\R$, where $0\in I\subseteq\R_-$, is \emph{invariant under horizontal max-differences} if, for every $\bfx\in I^n$ and every $c\in I$, we have
\begin{equation}\label{eq:HmaxDif}
f(\bfx)-f(\bfx\vee c)=f([\bfx]^c)-f([\bfx]^c\vee c).
\end{equation}

The following theorem provides a description of the class of functions which are comonotonically modular.

\begin{theorem}[{\cite{CouMar11b}}]\label{thm:wqer87we}
Assume that $I\ni 0$. For any function $f\colon I^n\to\R$, the following assertions are equivalent.
\begin{enumerate}
\item[$(i)$] $f$ is comonotonically modular.

\item[$(ii)$] $f|_{I_+^n}$ is comonotonically modular (or invariant under horizontal min-differences), $f|_{I_-^n}$ is
comonotonically modular (or invariant under horizontal max-differences), and we have $f(\bfx)+f(\bfzero)=f(\bfx^+)+f(-\bfx^-)$ for every $\bfx\in I^n$.

\item[$(iii)$] There exist $g\colon I_+^n\to\R$ and $h\colon I_-^n\to\R$ such that,
for every $\sigma\in\mathfrak{S}_X$ and every $\bfx\in I^n_{\sigma}$,
$$
f(\bfx) ~=~ f(\bfzero)+\sum_{i=1}^p\big
(h(x_{\sigma(i)}\mathbf{1}_{S_{\sigma}^{\downarrow}(i)})-h(x_{\sigma(i)}\mathbf{1}_{S_{\sigma}^{\downarrow}(i-1)})\big)+ \sum_{i=p+1}^n\big
(g(x_{\sigma(i)}\mathbf{1}_{S_{\sigma}^{\uparrow}(i)})-g(x_{\sigma(i)}\mathbf{1}_{S_{\sigma}^{\uparrow}(i+1)})\big),
$$
where $p\in\{0,\ldots,n\}$ is such that $x_{\sigma(p)}<0\leqslant x_{\sigma(p+1)}$, with the convention that $x_{\sigma(0)}=-\infty$ and $x_{\sigma(n+1)}=+\infty$. In this case, we can choose $g=f|_{I_+^n}$ and
$h=f|_{I_-^n}$.
\end{enumerate}
\end{theorem}
%
%

We finish this section with remarks on some properties subsumed by comonotonic modularity, namely
the following relaxations of maxitivity and minitivity properties.

Recall that a function $f\colon I^n\to\R$ is said to be \emph{maxitive} if
\begin{equation}\label{eq:sdf75max}
f(\bfx\vee\bfx')~=~f(\bfx)\vee f(\bfx'),\qquad\bfx,\bfx'\in I^n,
\end{equation}
 and it is said to be \emph{minitive}  if
\begin{equation}\label{eq:sdf75min}
f(\bfx\wedge\bfx')~=~f(\bfx)\wedge f(\bfx'),\qquad\bfx,\bfx'\in I^n.
\end{equation}
As in the case of modularity, maxitivity and minitivity give rise to noteworthy
decompositions of functions into maxima and minima, respectively, of $1$-variable functions.

In the context  of Sugeno integrals (see Section \ref{COMDinteg}), de Campos et al.~\cite{deCBol92} proposed
the following comonotonic variants of these properties.
A function $f\colon I^n\to\R$ is said to be
 \emph{comonotonic maxitive} (resp.\ \emph{comonotonic minitive}) if (\ref{eq:sdf75max}) (resp.\ (\ref{eq:sdf75min}))
 holds for any two comonotonic $n$-tuples $\bfx,\bfx'\in I^n$. It was shown in \cite{CouMar2} that any of these
 properties implies nondecreasing monotonicity, and it is not difficult to observe that comonotonic maxitivity together with comonotonic minitivity imply
 comonotonic modularity; the converse is not true (e.g., the arithmetic mean).

 Explicit descriptions of each one of these properties was given in \cite{CouMar09} for functions over bounded chains.
 For the sake of self-containment, we present these descriptions here. To this end, we now assume that $I=[a,b]\subseteq \R$, and
 for each $S\subseteq X$, we denote by ${\bf e}_S$  the $n$-tuple in $\{a,b\}^n$ whose $i$-th component
is $b$ if $i\in S$, and $a$ otherwise.

\begin{theorem}[{\cite{CouMar09}}]\label{thm:ComMaxMin}
Assume  $I=[a,b]\subseteq \R$. A function $f\colon I^n\to \R$ is comonotonic maxitive (resp.\ comonotonic minitive)
if and only if there exists a nondecreasing function $g\colon  I^n\to \R$  such that
$$
f(\bfx)=\bigvee_{S\subseteq X} g\Big({\bf e}_S\wedge\bigwedge_{i\in S}x_i\Big) \qquad \text{(resp. }
f(\bfx)=\bigwedge_{S\subseteq X} g\Big({\bf e}_{X\setminus S}\vee\bigvee_{i\in S}x_i\Big)\text{)}.
$$
In this case, we can choose $g=f$.
\end{theorem}

These descriptions are further refined in the following corollary.

\begin{corollary}\label{cor:ComMaxMin}
Assume  $I=[a,b]\subseteq \R$. For any function $f\colon I^n\to\R$, the following assertions are equivalent.
\begin{enumerate}
\item[$(i)$] $f$ is comonotonic maxitive (resp.\ comonotonic minitive).
\item[$(ii)$]
there are unary nondecreasing functions
$\varphi_S\colon I\to \R$ ($S\subseteq X$) such that
$$
f(\bfx)=\bigvee_{S\subseteq X}\varphi_S\Big(\bigwedge_{i\in S}x_i\Big)\qquad \text{(resp. }
 f(\bfx)=\bigwedge_{S\subseteq X}\varphi_S\Big(\bigvee_{i\in S}x_i\Big)\text{)}.
 $$
In this case, we can choose $\varphi_S(x)=f({\bf e}_S\wedge x)$ (resp.\ $\varphi_S(x)=f({\bf e}_{X\setminus S}\vee x)$)
for every $S\subseteq X$.

\item[$(iii)$]
for every $\sigma\in\mathfrak{S}_X$, there are nondecreasing functions $f^{\sigma}_i\colon I\to \R$ ($i=1,\ldots,n$)
such that, for every $\bfx\in I^n_{\sigma}$,
  $$
f(\bfx)=\bigvee_{i\in X}f^{\sigma}_i(x_{\sigma(i)})\qquad \text{(resp. }
f(\bfx)=\bigwedge_{i\in X}f^{\sigma}_i(x_{\sigma(i)})\text{)}.
  $$
In this case, we can choose $f^{\sigma}_i(x)=f({\bf e}_{S_{\sigma}^{\uparrow}(i)}\wedge x)$ (resp.\ $f^{\sigma}_i(x)=f({\bf e}_{S_{\sigma}^{\downarrow}(i-1)}\vee x)$).
\end{enumerate}
\end{corollary}

\begin{remark}
\begin{itemize}
\item[(i)]Note that the expressions provided in Theorem \ref{thm:ComMaxMin} and Corollary \ref{cor:ComMaxMin}
greatly differ from the additive form given in Theorem \ref{thm:wqer87we}.

\item[(ii)] An alternative description of comonotonic maxitive (resp.\ comonotonic minitive) functions
was obtained in Grabisch et al.~\cite[Ch.~2]{GraMarMesPap09}.
\end{itemize}
\end{remark}

\section{Classes of comonotonically modular integrals}\label{COMDinteg}

In this section we present axiomatizations of classes of functions that naturally generalize Choquet integrals
(e.g., signed Choquet integrals and symmetric signed Choquet integrals) by means of comonotonic modularity
and variants of homogeneity. From the analysis of the more stringent properties of comonotonic minitivity
and comonotonic maxitivity, we also present axiomatizations of classes of functions generalizing
Sugeno integrals.

\subsection{Comonotonically modular integrals generalizing Choquet integrals}

The following theorem provides an axiomatization of the class of $n$-variable signed Choquet integrals.

\begin{theorem}\label{thm:CModSCI}
Assume $[0,1]\subseteq I\subseteq \R_+$ or $[-1,1]\subseteq I$. Then a function $f\colon I^n\to\R$ is a signed Choquet integral if and only if the following conditions hold:
\begin{enumerate}
\item[$(i)$] $f$ is comonotonically modular.

\item[$(ii)$] $f(\bfzero)=0$ and $f(x\bfone_S)=\mathrm{sign}(x){\,}x{\,}f(\mathrm{sign}(x){\,}\bfone_S)$ for all $x\in I$ and $S\subseteq X$.

\item[$(iii)$] If $[-1,1]\subseteq I$, then $f(\bfone_{X\setminus S})=f(\bfone)+f(-\bfone_S)$ for all $S\subseteq X$.
\end{enumerate}
\end{theorem}

\begin{proof}
(Necessity) Assume that $f$ is a signed Choquet integral, $f=C_v$. Then condition (ii) is satisfied in view of Theorem~\ref{thm:saddsf687} and Remark~\ref{rem:saddsf687}. If $[-1,1]\subseteq I$, then by (\ref{eq:dual}) we have
$$
C_v(-\bfone_S) ~=~ -C_{v^d}(\bfone_S) ~=~ C_v(\bfone_{X\setminus S})-C_v(\bfone),
$$
which shows that condition (iii) is satisfied. Let us now show that condition (i) is also satisfied.
For every $\sigma\in\mathfrak{S}_X$ and every $\bfx\in\R^n_{\sigma}$,
setting $p\in\{0,\ldots,n\}$ such that $x_{\sigma(p)}<0\leqslant x_{\sigma(p+1)}$, by (\ref{eq:SignedChoquet2})
and conditions (iii) and (ii), we have
\begin{eqnarray*}
C_v(\bfx) &=& \sum_{i=1}^nx_{\sigma(i)}\,\big(C_v(\mathbf{1}_{S_{\sigma}^{\uparrow}(i)})-C_v(\mathbf{1}_{S_{\sigma}^{\uparrow}(i+1)})\big)\\
&=& \sum_{i=1}^px_{\sigma(i)}\,\big(C_v(-\mathbf{1}_{S_{\sigma}^{\downarrow}(i-1)})-C_v(-\mathbf{1}_{S_{\sigma}^{\downarrow}(i)})\big)%
+\sum_{i=p+1}^nx_{\sigma(i)}\,\big(C_v(\mathbf{1}_{S_{\sigma}^{\uparrow}(i)})-C_v(\mathbf{1}_{S_{\sigma}^{\uparrow}(i+1)})\big)\\
&=& \sum_{i=1}^p\big(C_v(x_{\sigma(i)}\,\mathbf{1}_{S_{\sigma}^{\downarrow}(i)})-C_v(x_{\sigma(i)}\,\mathbf{1}_{S_{\sigma}^{\downarrow}(i-1)})\big)%
+\sum_{i=p+1}^n\big(C_v(x_{\sigma(i)}\,\mathbf{1}_{S_{\sigma}^{\uparrow}(i)})-C_v(x_{\sigma(i)}\,\mathbf{1}_{S_{\sigma}^{\uparrow}(i+1)})\big)
\end{eqnarray*}
which shows that condition (iii) of Theorem~\ref{thm:wqer87we} is satisfied. Hence $C_v$ is comonotonically modular.

(Sufficiency) Assume that $f$ satisfies conditions (i)--(iii). By condition (iii) of Theorem~\ref{thm:wqer87we} and conditions (ii) and (iii), for every $\sigma\in\mathfrak{S}_X$ and every $\bfx\in\R^n_{\sigma}$ we have
\begin{eqnarray*}
f(\bfx) &=& \sum_{i=1}^p\big
(f(x_{\sigma(i)}\mathbf{1}_{S_{\sigma}^{\downarrow}(i)})-f(x_{\sigma(i)}\mathbf{1}_{S_{\sigma}^{\downarrow}(i-1)})\big)%
+\sum_{i=p+1}^n\big
(f(x_{\sigma(i)}\mathbf{1}_{S_{\sigma}^{\uparrow}(i)})-f(x_{\sigma(i)}\mathbf{1}_{S_{\sigma}^{\uparrow}(i+1)})\big)\\
&=& \sum_{i=1}^px_{\sigma(i)}\,\big(f(-\mathbf{1}_{S_{\sigma}^{\downarrow}(i-1)})-f(-\mathbf{1}_{S_{\sigma}^{\downarrow}(i)})\big)%
+ \sum_{i=p+1}^nx_{\sigma(i)}\,\big
(f(\mathbf{1}_{S_{\sigma}^{\uparrow}(i)})-f(\mathbf{1}_{S_{\sigma}^{\uparrow}(i+1)})\big)\\
&=& \sum_{i=1}^nx_{\sigma(i)}\,\big
(f(\mathbf{1}_{S_{\sigma}^{\uparrow}(i)})-f(\mathbf{1}_{S_{\sigma}^{\uparrow}(i+1)})\big)
\end{eqnarray*}
which, combined with (\ref{eq:SignedChoquet2}), shows that $f$ is a signed Choquet integral.
\end{proof}

\begin{remark}
Condition (iii) of Theorem~\ref{thm:CModSCI} is necessary. Indeed, the function $f(\bfx)=C_v(\bfx^+)$ satisfies conditions (i) and (ii) but fails to satisfy condition (iii).
\end{remark}

\begin{theorem}\label{thm:CModSSCI}
Assume $I$ is centered at $0$ with $[-1,1]\subseteq I$. Then a function $f\colon I^n\to\R$ is a symmetric signed Choquet integral if and only if the following conditions hold:
\begin{enumerate}
\item[$(i)$] $f$ is comonotonically modular.

\item[$(ii)$] $f(x\bfone_S)=x{\,}f(\bfone_S)$ for all $x\in I$ and $S\subseteq X$.
\end{enumerate}
\end{theorem}

\begin{proof}
(Necessity) Assume that $f$ is a symmetric signed Choquet integral, $f=\check{C}_v$. Then condition (ii) is satisfied in view of Theorem~\ref{thm:saddsf687a} and Remark~\ref{rem:saddsf687a}. Let us now show that condition (i) is also satisfied. For every $\sigma\in\mathfrak{S}_X$ and every $\bfx\in\R^n_{\sigma}$, setting $p\in\{0,\ldots,n\}$ such that $x_{\sigma(p)}<0\leqslant x_{\sigma(p+1)}$, by (\ref{eq:sdfsfd65dsf}) and condition (ii), we have
\begin{eqnarray*}
C_v(\bfx) &=& \sum_{i=1}^px_{\sigma(i)}\,\big(C_v(\mathbf{1}_{S_{\sigma}^{\downarrow}(i)})-C_v(\mathbf{1}_{S_{\sigma}^{\downarrow}(i-1)})\big)%
+\sum_{i=p+1}^nx_{\sigma(i)}\,\big(C_v(\mathbf{1}_{S_{\sigma}^{\uparrow}(i)})-C_v(\mathbf{1}_{S_{\sigma}^{\uparrow}(i+1)})\big)\\
&=& \sum_{i=1}^p\big(C_v(x_{\sigma(i)}\,\mathbf{1}_{S_{\sigma}^{\downarrow}(i)})-C_v(x_{\sigma(i)}\,\mathbf{1}_{S_{\sigma}^{\downarrow}(i-1)})\big)%
+\sum_{i=p+1}^n\big(C_v(x_{\sigma(i)}\,\mathbf{1}_{S_{\sigma}^{\uparrow}(i)})-C_v(x_{\sigma(i)}\,\mathbf{1}_{S_{\sigma}^{\uparrow}(i+1)})\big)
\end{eqnarray*}
which shows that condition (iii) of Theorem~\ref{thm:wqer87we} is satisfied. Hence $C_v$ is comonotonically modular.

(Sufficiency) Assume that $f$ satisfies conditions (i) and (ii). By condition (iii) of Theorem~\ref{thm:wqer87we} and condition (ii), for every $\sigma\in\mathfrak{S}_X$ and every $\bfx\in\R^n_{\sigma}$ we have
\begin{eqnarray*}
f(\bfx) &=& \sum_{i=1}^p\big
(f(x_{\sigma(i)}\mathbf{1}_{S_{\sigma}^{\downarrow}(i)})-f(x_{\sigma(i)}\mathbf{1}_{S_{\sigma}^{\downarrow}(i-1)})\big)%
+ \sum_{i=p+1}^n\big
(f(x_{\sigma(i)}\mathbf{1}_{S_{\sigma}^{\uparrow}(i)})-f(x_{\sigma(i)}\mathbf{1}_{S_{\sigma}^{\uparrow}(i+1)})\big)\\
&=& \sum_{i=1}^px_{\sigma(i)}\,\big(f(\mathbf{1}_{S_{\sigma}^{\downarrow}(i)})-f(\mathbf{1}_{S_{\sigma}^{\downarrow}(i-1)})\big)%
+\sum_{i=p+1}^nx_{\sigma(i)}\,\big(f(\mathbf{1}_{S_{\sigma}^{\uparrow}(i)})-f(\mathbf{1}_{S_{\sigma}^{\uparrow}(i+1)})\big)
\end{eqnarray*}
which, combined with (\ref{eq:sdfsfd65dsf}), shows that $f$ is a symmetric signed Choquet integral.
\end{proof}

The authors \cite{CouMar11b} showed that comonotonically modular functions also include the class of
signed quasi-Choquet integrals on intervals of the forms $I_+$ and $I_-$ and
the class of symmetric signed quasi-Choquet integrals on intervals $I$ centered at the origin.

\begin{definition}
Assume $I\ni 0$ and let $v$ be a signed capacity on $X$. A \emph{signed quasi-Choquet integral} with respect to $v$ is a function $f\colon I^n\to\R$ defined as $f(\bfx)=C_v(\varphi(x_1),\ldots,\varphi(x_n))$, where $\varphi\colon I\to\R$ is a nondecreasing function satisfying $\varphi(0)=0$.
\end{definition}

We now recall axiomatizations of the class of $n$-variable signed quasi-Choquet integrals on $I_+$ and $I_-$ by means of comonotonic modularity and variants of homogeneity.

\begin{theorem}[{\cite{CouMar11b}}]\label{thm:sdfa5dsa76}
Assume $[0,1]\subseteq I\subseteq\R_+$ (resp.\ $[-1,0]\subseteq I\subseteq\R_-$) and let $f\colon I^n\to\R$ be a nonconstant function such that $f(\bfzero)=0$. Then the following assertions are equivalent.
\begin{enumerate}
\item[$(i)$] $f$ is a signed quasi-Choquet integral and there exists $S\subseteq X$ such that $f(\mathbf{1}_S)\neq 0$ (resp.\ $f(-\mathbf{1}_S)\neq 0$).

\item[$(ii)$] $f|_{I_+^n}$ is comonotonically modular (or invariant under horizontal min-differences), $f|_{I_-^n}$ is
comonotonically modular (or invariant under horizontal max-differences), and there exists a nondecreasing function $\varphi\colon I\to\R$ satisfying $\varphi(0)=0$ such that $f(x\bfone_S)=\mathrm{sign}(x){\,}\varphi(x){\,}f(\mathrm{sign}(x){\,}\bfone_S)$ for every $x\in I$ and every $S\subseteq X$.
%
\end{enumerate}
\end{theorem}

\begin{remark}
If $I=[0,1]$ (resp.\ $I=[-1,0]$), then the ``nonconstant'' assumption and the second condition in assertion $(i)$ of
Theorem~\ref{thm:sdfa5dsa76} can be dropped off.
\end{remark}

The extension of Theorem~\ref{thm:sdfa5dsa76} to functions on intervals $I$ centered at $0$ and containing $[-1,1]$ remains an interesting open problem.

We now recall the axiomatization obtained by the authors of the class of $n$-variable symmetric signed quasi-Choquet integrals.

\begin{definition}
Assume $I$ is centered at $0$ and let $v$ be a signed capacity on $X$. A \emph{symmetric signed quasi-Choquet integral} with respect to $v$ is a function $f\colon I^n\to\R$ defined as $f(\bfx)=\check{C}_v(\varphi(x_1),\ldots,\varphi(x_n))$, where $\varphi\colon I\to\R$ is a nondecreasing odd function.
\end{definition}

\begin{theorem}[{\cite{CouMar11b}}]\label{thm:CanonicalSymQuasi-Lovasz}
Assume that $I$ is centered at $0$ with $[-1,1]\subseteq I$ and let $f\colon I^n\to\R$ be a function such that $f|_{I_+^n}$ or $f|_{I_-^n}$ is
nonconstant and $f(\bfzero)=0$. Then the following assertions are equivalent.
\begin{enumerate}
\item[$(i)$] $f$ is a symmetric signed quasi-Choquet integral and there exists $S\subseteq X$ such that
$f(\mathbf{1}_S)\neq 0$.

\item[$(ii)$] $f$ is comonotonically modular and there exists a nondecreasing odd function $\varphi\colon I\to\R$ such that $f(x\bfone_S)=\varphi(x){\,}f(\bfone_S)$ for every $x\in I$ and every $S\subseteq X$.
  %
\end{enumerate}
\end{theorem}

\begin{remark}
If $I=[-1,1]$, then the ``nonconstant'' assumption and the second condition in assertion $(i)$ of Theorem~\ref{thm:CanonicalSymQuasi-Lovasz} can
be dropped off.
\end{remark}

\subsection{Comonotonically modular integrals generalizing Sugeno integrals}

In this subsection we consider natural extensions of the $n$-variable Sugeno integrals on a bounded real interval $I=[a,b]$.
By an $I$-\emph{valued capacity} on $X$ we mean an order preserving mapping $\mu\colon 2^X\to I$ such that
$\mu(\varnothing)=a$ and $\mu(X)=b$.

\begin{definition}
Assume that $I=[a,b]$. The \emph{Sugeno integral with respect to an $I$-valued capacity $\mu$} on $X$ is the function
$\mathcal{S}_\mu\colon I^n\to I$ defined as
$$
\mathcal{S}_\mu(\bfx)~=~ \bigvee_{i\in X}x_{\sigma(i)}\wedge \mu(S^{\uparrow}_\sigma(i)),\qquad \bfx \in I^n_\sigma,~\sigma\in\mathfrak{S}_X.
$$
\end{definition}

As the following proposition suggests, Sugeno integrals can be viewed as idempotent ``lattice  polynomial functions''
(see \cite{Marc}).

\begin{proposition}
Assume that $I=[a,b]$. A function $f\colon I^n\to I$ is a Sugeno integral if and only if $f(\mathbf{e}_{\varnothing})=a$, $f(\mathbf{e}_X)=b$, and for every $\bfx\in I^n$
 $$
 f(\bfx)~=~\bigvee_{S\subseteq X}f(\mathbf{e}_S)\wedge\bigwedge_{i\in S}x_i{\,}.
 $$
\end{proposition}

As mentioned, the properties of comonotonic maxitivity and comonotonic minitivity were introduced by
de Campos et al.\ in \cite{deCBol92} to axiomatize the class of Sugeno integrals. However,
without further assumptions, they define a wider class of functions that we now define.

\begin{definition}
Assume that $I=[a,b]$ and $J=[c,d]$ are real intervals and let $\mu$ be an $I$-valued capacity on $X$.
A \emph{quasi-Sugeno integral} with respect to $\mu$ is a function $f\colon J^n\to I$
defined by $f(\bfx)=\mathcal{S}_{\mu}(\varphi(x_1),\ldots,\varphi(x_n))$, where $\varphi\colon J\to I$ is a nondecreasing function.
\end{definition}

Using Proposition 11 and Corollary 17 in \cite{CouMar10}, we obtain the
following axiomatization of the class of quasi-Sugeno integrals.

\begin{theorem}\label{thm:quasi-SugMaxMin}
Let $I=[a,b]$ and $J=[c,d]$ be real intervals and consider a function $f\colon J^n\to I$.
The following assertions are equivalent.
\begin{enumerate}
\item[$(i)$] $f$ is a quasi-Sugeno integral.
\item[$(ii)$] $f$ is comonotonically maxitive and comonotonically minitive.
\item[$(iii)$] $f$ is nondecreasing, and there exists a nondecreasing function $\varphi\colon J\to I$
such that for every $\bfx\in J^n$ and $r\in J$, we have
\begin{equation}\label{minHom2}
f(r\vee \bfx) =  \varphi(r) \vee f(\bfx) \qquad \text{and}\qquad
f(r\wedge \bfx) = \varphi(r)\wedge f(\bfx),
\end{equation}
where $r\vee \bfx$ (resp.\ $r\wedge \bfx$) is the $n$-tuple whose $i$th component is $r\vee x_i$ (resp.\ $r\wedge x_i$). In this case, $\varphi$ can be chosen as $\varphi(x)=f(x,\ldots,x)$.
\end{enumerate}
\end{theorem}

\begin{remark}
The two conditions given in (\ref{minHom2}) are referred to in \cite{CouMar10} as  \emph{quasi-max homogeneity} and
\emph{quasi-min homogeneity}, respectively.
\end{remark}

As observed at the end of the previous section, condition $(ii)$ (and hence $(i)$ or $(iii)$) of Theorem \ref{thm:quasi-SugMaxMin}
implies comonotonic modularity. As the following result shows, the converse is true whenever $f$ is nondecreasing
and verifies any of the following weaker variants of quasi-max homogeneity and quasi-min homogeneity:
\begin{eqnarray}
f(x\vee\mathbf{e}_{S}) &=& f(x,\ldots,x)\vee f(\mathbf{e}_S),\qquad x\in J,~S\subseteq  X,\label{weakMaxHom}\\
f(x\wedge\mathbf{e}_{S}) &=& f(x,\ldots,x)\wedge f(\mathbf{e}_S),\qquad x\in J,~S\subseteq  X.\label{weakMinHom}
\end{eqnarray}

\begin{theorem}\label{thm:quasi-Mod}
Let $I=[a,b]$ and $J=[c,d]$ be real intervals and consider a function $f\colon J^n\to I$. The following conditions are equivalent.
\begin{enumerate}
\item[$(i)$] $f$ is a quasi-Sugeno integral, $f(\bfx)=\mathcal{S}_{\mu}(\varphi(x_1),\ldots,\varphi(x_n))$, where $\varphi(x)=f(x,\ldots,x)$.

\item[$(ii)$] $f$ is a quasi-Sugeno integral.

\item[$(iii)$] $f$ is comonotonically modular, nondecreasing, and satisfies (\ref{weakMaxHom}) or (\ref{weakMinHom}).

\item[$(iv)$] $f$ is nondecreasing and satisfies (\ref{weakMaxHom}) and (\ref{weakMinHom}).
\end{enumerate}
\end{theorem}

\begin{proof}
$(i)\Rightarrow (ii)$ Trivial.

$(ii)\Rightarrow (iii)$ Follows from Theorem~\ref{thm:quasi-SugMaxMin}.

$(iii)\Rightarrow (iv)$ Suppose that $f$ is comonotonically modular and satisfies (\ref{weakMaxHom}). Then,
\begin{eqnarray*}
f(x\wedge\mathbf{e}_S) &=& f(x,\ldots,x) + f(\mathbf{e}_S) - f(x\vee\mathbf{e}_S)\\
&=& f(x,\ldots,x) + f(\mathbf{e}_S) - f(x,\ldots,x)\vee f(\mathbf{e}_S)  ~=~ f(x,\ldots,x)\wedge f(\mathbf{e}_S).
\end{eqnarray*}
Hence $f$ satisfies (\ref{weakMinHom}). The other case can be dealt with dually.

$(iv)\Rightarrow (i)$ Define $\varphi(x)=f(x,\ldots,x)$. By nondecreasing monotonicity and (\ref{weakMinHom}), for every $S\subseteq X$ we have
\begin{eqnarray*}
f(\vect{x}) ~ \geqslant ~ f\Big(\vect{e}_S \wedge \bigwedge_{i\in S}x_i\Big) ~=~
f(\vect{e}_S) \wedge \varphi\Big(\bigwedge_{i\in S}x_i\Big)~=~f(\vect{e}_S) \wedge \bigwedge_{i\in S}\varphi(x_i)
\end{eqnarray*}
and thus $f(\vect{x})\geqslant \bigvee_{S\subseteq X} f(\vect{e}_S) \wedge  \bigwedge_{i\in S}\varphi(x_i)$.
To complete the proof, it is enough to establish the converse inequality. Let $S^*\subseteq X$ be such that
$f(\vect{e}_{S^*}) \wedge \bigwedge_{i\in S^*}\varphi(x_i)$ is maximum. Define
$$
T=\Big\{j\in X: \varphi(x_j)\leqslant f(\vect{e}_{S^*}) \wedge\bigwedge_{i\in S^*}\varphi(x_i)\Big\}.
$$
We claim that $T\neq\varnothing$. Suppose this is not true, that is,
$\varphi(x_j)> f(\vect{e}_{S^*}) \wedge\bigwedge_{i\in S^*}\varphi(x_i)$
for every $j\in X$. Then, by nondecreasing monotonicity, we have
$f(\vect{e}_{X})\geqslant f(\vect{e}_{S^*})$, and since
$f(\vect{e}_{X})\geqslant \bigwedge_{i\in X}\varphi(x_i)$,
$$
f(\vect{e}_{X}) \wedge \bigwedge_{i\in X}\varphi(x_i) ~>~ f(\vect{e}_{S^*}) \wedge\bigwedge_{i\in S^*}\varphi(x_i)
$$
which contradicts the definition of $S^*$. Thus $T\neq \varnothing$.

Now, by nondecreasing monotonicity and (\ref{weakMaxHom})  we have
$$
f(\vect{x}) ~ \leqslant ~  f\Big(\vect{e}_{X\setminus T}\vee\bigvee_{j\in T}x_j\Big) ~=~ f(\vect{e}_{X\setminus T})\vee\varphi\Big(\bigvee_{j\in T}x_j\Big) ~=~  f(\vect{e}_{X\setminus T})\vee\bigvee_{j\in T}\varphi(x_j) ~=~  f(\vect{e}_{X\setminus T}).
$$
Indeed, we have $\varphi(x_j)\leqslant f(\vect{x})$ for every $j\in T$ and $\vect{x} \leqslant
\vect{e}_{X\setminus T}\vee\bigvee_{j\in T}x_j$.

Note that $f(\vect{e}_{X\setminus T})\leqslant f(\vect{e}_{S^*}) \wedge \bigwedge_{i\in S^*}\varphi(x_i)$ since
 otherwise, by definition of $T$, we would have
$$
f(\vect{e}_{X\setminus T})\wedge \bigwedge_{i\in X\setminus T}\varphi(x_i) ~ >
~ f(\vect{e}_{S^*}) \wedge \bigwedge_{i\in S^*}\varphi(x_i),
$$
again contradicting the definition of $S^*$. Finally,
$$
f(\vect{x}) ~ \leqslant ~f(\vect{e}_{S^*}) \wedge \bigwedge_{i\in S^*}\varphi(x_i) ~=~
\bigvee_{S\subseteq X} f(\vect{e}_S) \wedge
 \bigwedge_{i\in S}\varphi(x_i),
$$
and the proof is thus complete.
\end{proof}

\begin{remark}
An axiomatization of the class of Sugeno integrals based on comonotonic modularity can be obtained from Theorems~\ref{thm:quasi-SugMaxMin} and \ref{thm:quasi-Mod} by adding the idempotency property.
\end{remark}

\section{Conclusion}

In this paper we analyzed comonotonic modularity as a feature common to many well-known discrete integrals.
In doing so, we established its relation to many other noteworthy aggregation properties such as comonotonic relaxations of additivity, maxitivity and minitivity. In fact, the latter become equivalent in presence of comonotonic modularity. As a by-product we immediately see that, e.g., the so-called discrete Shilkret integral lies outside the class of comonotonic modular functions since this integral is comonotonically maxitive but not comonotonically minitive.

Albeit such an example, the class of comonotonically modular functions is
rather vast and includes several important extensions of the Choquet and Sugeno integrals. The results presented in Section~\ref{COMDinteg} seem to indicate that suitable variants of homogeneity suffice to distinguish and fully describe these extensions. This naturally asks for an exhaustive study of homogeneity-like properties, which may lead to a complete classification of all subclasses of comonotonically modular functions.

Another question that still eludes us is the relation between the additive forms given by comonotonic modularity and the max-min forms. As shown in Theorem~\ref{thm:quasi-Mod}, the latter are particular instances of the former; in fact, proof of Theorem~\ref{thm:quasi-Mod} provides a procedure to construct max-min representations of comonotonically modular functions, whenever they exist. However, we were not able to present a direct translation between the two. This remains as a relevant open question since its answer will inevitably provide a better understanding of the synergy between these intrinsically different normal forms.

\section*{Acknowledgments}

This research is partly supported by the internal research project F1R-MTH-PUL-12RDO2 of the University of Luxembourg.

\end{document}